\documentclass[11pt]{amsart}
\usepackage[colorlinks,citecolor=red,pagebackref,hypertexnames=false]{hyperref}

\makeatletter
\def\@tocline#1#2#3#4#5#6#7{\relax
  \ifnum #1>\c@tocdepth 
  \else
    \par \addpenalty\@secpenalty\addvspace{#2}%
    \begingroup \hyphenpenalty\@M
    \@ifempty{#4}{%
      \@tempdima\csname r@tocindent\number#1\endcsname\relax
    }{%
      \@tempdima#4\relax
    }%
    \parindent\z@ \leftskip#3\relax \advance\leftskip\@tempdima\relax
    \rightskip\@pnumwidth plus4em \parfillskip-\@pnumwidth
    #5\leavevmode\hskip-\@tempdima
      \ifcase #1
       \or\or \hskip 1em \or \hskip 2em \else \hskip 3em \fi%
      #6\nobreak\relax
    \dotfill\hbox to\@pnumwidth{\@tocpagenum{#7}}\par
    \nobreak
    \endgroup
  \fi}
\makeatother

\usepackage{tikz,amsthm,amsmath,amstext,amssymb,amscd,epsfig,euscript, mathrsfs, dsfont,pspicture,multicol, mathtools,graphpap,graphics,graphicx,times,enumerate,subfig,sidecap,wrapfig,color}

\usepackage{ hyperref}
\usepackage{enumerate}

 \numberwithin{equation}{section}

\def\om{{\Omega}}
\def\hm{{\omega}}

\def\R{{\mathbb{R}}}

\def\bN{{\mathbb{N}}}

\def\G{{\Gamma}}

\def\K{{\mathcal{T}}}
\def\P{{\mathcal{P}}}

\def\HH{{\mathcal{H}}}

\def\ve{\varepsilon}

\renewcommand{\d}{{\partial}}

\DeclareMathOperator{\diam}{diam}
\def\dist{\mathop\mathrm{dist}} 						
\def\loc{\mathop\mathrm{loc}}						

\def\XXint#1#2#3{{\setbox0=\hbox{$#1{#2#3}{\int}$ }
\vcenter{\hbox{$#2#3$ }}\kern-.58\wd0}}

\newcommand{\divv}{{\text{{\rm div}}}}

\theoremstyle{plain}
\newtheorem{theorem}{Theorem}

\newtheorem{corollary}[theorem]{Corollary}

\newtheorem{lemma}[theorem]{Lemma}

\newtheorem{proposition}[theorem]{Proposition}

\theoremstyle{definition}

\newtheorem{definition}[theorem]{Definition}
\newtheorem{remark}[theorem]{Remark}

\numberwithin{equation}{section}
\numberwithin{theorem}{section}

\newcommand{\vv}{\vspace{2mm}}
\newcommand{\vvv}{\vspace{4mm}}

  \DeclareFontFamily{U}{mathb}{\hyphenchar\font45} 
\DeclareFontShape{U}{mathb}{m}{n}{
      <5> <6> <7> <8> <9> <10> gen * mathb
      <10.95> mathb10 <12> <14.4> <17.28> <20.74> <24.88> mathb12
      }{}
\DeclareSymbolFont{mathb}{U}{mathb}{m}{n}

\DeclareMathSymbol{\toitself}      {3}{mathb}{"FD}  

\begin{document}

\title[Approximate tangets and harmonic measure]{Approximate tangents, harmonic measure, and domains with rectifiable boundaries}

\author{Mihalis Mourgoglou}
\address{Departamento de Matem\'aticas, Universidad del Pa\' is Vasco, Barrio Sarriena s/n 48940 Leioa, Spain and\\
Ikerbasque, Basque Foundation for Science, Bilbao, Spain.}
\email{michail.mourgoglou@ehu.eus}

\thanks{M.M. was supported  by IKERBASQUE and partially supported by the grant MTM-2017-82160-C2-2-P of the Ministerio de Econom\'ia y Competitividad (Spain), and by  IT-1247-19 (Basque Government). 
}
\maketitle

\begin{abstract}
Let $\Omega \subset \R^{n+1}$, $n \geq 1$, be an open and connected set.  Set $\mathcal{T}_n$ to be  the set of points $\xi \in  \d \om$   so that there exists an approximate tangent $n$-plane  for $\d\om$ at $\xi$ and $\d\om$ satisfies the weak lower Ahlfors-David $n$-regularity condition at $\xi$. We first show  that $\mathcal{T}_n$ can be covered by a countable union of  boundaries  of bounded Lipschitz domains. Then, letting $\d^\star \om$ be a subset  of $\mathcal{T}_n$ where $\om$ satisfies an appropriate thickness condition,  we  prove that $\d^\star \om$ can be covered by a countable union of boundaries of bounded Lipschitz domains  contained in $\om$. As a corollary we obtain that if $\om$ has locally finite perimeter, $\d\om$ is  weakly lower Ahlfors-David $n$-regular, and the measure-theoretic boundary  coincides with the topological boundary of $\om$ up to a set of $\HH^n$-measure zero, then $\partial \om$ can be covered, up to a set of $\HH^n$-measure zero, by a countable union of   boundaries of bounded Lipschitz domains that are contained in $\om$. This implies  that in such  domains, $\HH^n|_{\d\om}$ is absolutely continuous with respect to harmonic measure. 
\end{abstract}
\tableofcontents

\section{Introduction}\label{sec:intro}

There is a strong connection between  rectifiability of the boundary of a domain $\om$ in the Euclidean space $\R^{n+1}$ and  absolute continuity of harmonic measure defined in $\om$ with respect to the $n$-Hausdorff measure on $\d \om$. In the past few years there has been a renaissance of new results in this field shedding light on long-standing related questions. Let us now briefly review the history of work in this area.


 As early as 1916, F. and M. Riesz \cite{RR} showed that for simply connected planar domains that are bounded by a Jordan curve and whose boundary has finite length, harmonic measure and arc-length are mutually absolutely continuous.  Lavrent'ev \cite{Lav} quantified this result by demonstrating that in a simply connected domain in the complex plane, bounded by a chord-arc curve,  harmonic measure is in the $A_\infty$ class of Muckehoupt weights. McMillan showed in \cite[Theorem 2]{McM69} that for bounded simply connected domains $\Omega\subset \mathbb C$, harmonic measure $\omega_\Omega$ and $\HH^1$ measure are mutually absolutely continuous  on the set of cone points. A local version of F. and M. Riesz theorem was obtained by Bishop and Jones \cite{BJ} where they showed that if $\Omega$ is a simply connected planar domain and $\Gamma$ is a curve of finite length, then $\omega\ll \HH^{1}$ on $\d\Omega\cap \Gamma$, where $\omega$ stands for the harmonic measure. They also constructed an infinitely connected planar domain $\Omega$ whose boundary $\d \om$ is uniformly rectifiable  but harmonic measure is not absolutely continuous to arc length (thus showing that some sort of connectivity  is required). 

In higher dimensions, the situation is more complicated. The obvious generalization of F. and M. Riesz theorem to higher dimensions is false due to examples of Wu and Ziemer: they constructed topological two-spheres in $\mathbb{R}^{3}$ with boundaries of finite Hausdorff measure $\HH^{2}$ where either harmonic measure is not absolutely continuous with respect to $\HH^{2}$ \cite{Wu} or $\HH^{2}$ is not absolutely continuous with respect to harmonic measure \cite{Z}. Dahlberg \cite{Da} proved that in a  Lipschitz domain,  harmonic measure is in the $A_\infty$ class of Muckenhoupt weights with respect to the $n$-Hausdorff measure restricted to the boundary. The same result was proved by David and Jerison in \cite{DJ90}  under the assumptions that $\Omega\subset \mathbb{R}^{n+1}$ is an NTA domain (for the definition see \cite{JK}) and $\d\Omega$ is Ahlfors-David regular. Azzam, Hofmann, Martell, Nystr\"om and Toro \cite{AHMNT} showed that any uniform domain with uniformly rectifiable boundary is an NTA domain and thus, $\omega \in A_\infty$ by \cite{DJ90} (a direct proof of the $A_\infty$-equivalence between $\omega$ and $\HH^n|_{\d \Omega}$ in this case was given earlier by Hofmann and Martell \cite{HM12}). For a local version of this result  by Azzam see \cite{Az1}. Badger \cite{Badger12} showed that if one merely assumes $\HH^{n}|_{\d\Omega}$ is locally finite and $\Omega\subset \mathbb{R}^{n+1}$ is NTA, then we still have $\HH^{n}|_{\d\Omega} \ll \omega$ (see also the work of Azzam in \cite{Az2}).  The author of the current manuscript obtained in \cite{Mo} a refinement of Badger's result by  proving that for a uniform domain of locally finite perimeter, with rectifiable boundary satisfying the lower ADR condition, surface measure is again absolutely continuous with respect to harmonic measure. Independently, Akman, Badger, Hofmann, and Martell \cite{ABHM15}  (among others) obtained the same result assuming  both upper and lower ADR. Later, Akman, Azzam, and the author \cite{AAM} studied absolute continuity of harmonic measure with respect to surface measure on domains $\Omega$ that have large complements and showed that if $\Gamma\subset  \mathbb R^{n+1}$ is Ahlfors-David $n$-regular and splits $ \mathbb R^{n+1}$ into two NTA domains, then $\omega_{\Omega}\ll \HH^{n}$ on $\Gamma\cap \partial\Omega$. This result is a natural generalization of a theorem of Wu in \cite{Wu}.

{
While this paper was in preparation, Akman, Bortz, Hofmann, and Martell \cite{ABHM16} proved, among others, a slightly more general version of Theorem \ref{thm:w-abscont-rect} under more general assumptions than the ones  we had originally  imposed using that rectifiable sets are $n$-linearly approximable (instead of having approximate tangents $\HH^n$-a.e.).  It is worth mentioning that their results are also local in nature. In addition, they constructed examples of domains to show that the results for harmonic measure were optimal. The present paper provides an alternative method to obtain a slightly weaker version of the  main theorems of  \cite{ABHM16}  and that is how it should be considered. 
}
\vv

Let us now state our results .
\begin{theorem}\label{thm:structure}
Let $E \subset \R^{n+1}$, $n \geq 1$ be a closed set and $s \in (0,1/3)$ be fixed. Set $\mathcal{T}_m(E) \subset E$ to be the set of all points $x \in E$ for which
\begin{enumerate}
\item  there exists an $s$-approximate tangent $m$-plane $V_x$ for $E$ at $x$ and 
\item  $E$ satisfies the weak lower Ahlfors-David $m$-regularity condition at $x$.
\end{enumerate}
Then there exists a countable collection of  bounded Lipschitz graphs  $\{\Gamma_j\}_{j\geq 1}$ so that $\K_m(E)  \subset \cup_{j\geq 1} \Gamma_j$. In particular, $\K_m(E) $ is $m$-rectifiable. 
\end{theorem}

{ For the definitions of  weak lower Ahlfors-David $m$-regularity and approximate tangents, we refer to Definitions \ref{def:WLADR} and  \ref{def:appr.tang.pl}. We emphasize that we did not assume that $\HH^m|_E$, the $m$-Hausdorff measure on $E$, is locally finite.}

\vv

\begin{theorem}\label{thm:Lip-dom}
Let $E \subset \R^{n+1}$, $n \geq 1$ be a closed set and $s \in (0,1/\sqrt{90})$ be fixed. If $\K_{n}(E) $ is as in Theorem \ref{thm:structure}, then there exist two countable collections of bounded Lipschitz domains $\{\om_j\}_{j \geq 1}$ such that  $\om_j^+ \cap \om_j^-= \emptyset$, $\K_{n}(E)  \cap \d \om^+_j =\K_{n}(E)  \cap \d \om^-_j$, and $\K_{n}(E)  \subset \cup_j \d \om^\pm_j $.
\end{theorem}

\vv

{
 Let $\Omega \subset \R^{n+1}$, $n \geq 1$, be an open set. Let $\xi \in \d\om$ and $\nu \in \mathbb{S}^n=\d B(0,1)$, and denote by
\begin{align*}
  H^+(\xi, \nu) &= \{ x \in \R^d: (x-\xi) \cdot \nu > 0\},
\end{align*}
 the upper half-space defined by $\nu$. If $V_\xi $ is an approximate tangent $n$-plane for $\d \om$ at $\xi  \in \d\om$, we define by $\nu_\xi ^+,\nu_\xi ^-  \in \mathbb{S}^{n}$ to be  the two  vectors in the unit sphere  that are orthogonal to $V_\xi $.

\vv 
 
 \begin{definition}\label{def:thick}
 We say that $\xi  \in \mathcal \d^\star \om$  if $\xi \in \K_n(\d\om)$ and either 
\begin{equation}\label{eq:thick1}
\liminf_{r \to 0} \frac{|B(\xi,r) \cap H^+(\xi, \nu_\xi^+) \setminus \Omega |}{|B(\xi,r)|}=0,
\end{equation}
or
\begin{equation}\label{eq:thick2}
 \liminf_{r \to 0} \frac{|B(\xi,r) \cap H^+(\xi, \nu_\xi^-) \setminus \Omega |}{|B(\xi,r)|}=0.
\end{equation}
\end{definition}
Here $| \cdot|$ stands for the $(n+1)$-dimensional Lebesgue measure.

\vv

\begin{theorem}\label{thm:Lip-dom-int}
 Let $\Omega \subset \R^{n+1}$, $n \geq 1$, be an open and connected set and fix $s \in (0,1/\sqrt{90})$. If  $\d^\star \om$ is as in Definition \ref{def:thick}, then there exists a countable collection of bounded Lipschitz domains $\{\om_j\}_{j \geq 1}$ such that  $ \om_j \subset \om $ and $ \d^\star \om\subset \cup_j \d \om_j $. As a consequence we obtain that  $\HH^n \ll \hm_{\om}^x$ on $\d^\star\om$, for any $x \in \om$, where $\hm^x_{\om}$ stands for the harmonic measure for $\om$ with pole at $x \in \om$.
\end{theorem}

\vv

An immediate consequence of Theorem \ref{thm:Lip-dom-int} is the following:
\begin{theorem}\label{thm:w-abscont-rect}
Let $\Omega \subset \R^{n+1}$, $n \geq 1$, be an open and connected set with $n$-rectifiable and weakly lower Ahlfors-David $n$-regular boundary $\d \om$. Let also  $\HH^{n}|_{\d \om}$ be locally finite and assume that
$$\HH^{n}(\d \om \setminus \d^\star \om)=0,$$ 
where $\d^\star \om$ is as in Definition \ref{def:thick}. Then, there exists a countable collection of bounded Lipschitz domains $\{\om_j\}_{j \geq 1}$ contained in $\om$ such that $\d\om \subset  \cup_j \d \om_j \cup N$, for some $N \subset \d \om$ with $\HH^{n}(N)=0$. Moreover,  $\HH^n|_{\d \om} \ll \hm^x_{\om}$  for any $x \in \om$.
\end{theorem}

\vv

Finally,  by Theorem \ref{thm:w-abscont-rect}, we obtain the following corollary.
\begin{corollary}\label{cor:w-abscont-rect}
Let $\Omega \subset \R^{n+1}$, $n \geq 1$, be an open and connected set of locally finite perimeter. Let also $\d\om$ be weakly lower Ahlfors-David $n$-regular and $\HH^{n}( \d \om \setminus \d_* \om)=0,$ where $\d_* \om$ stands for the measure-theoretic boundary of $\om$. Then, there exists a countable collection of bounded Lipschitz domains $\{\om_j\}_{j \geq 1}$ contained in $\om$ such that $\d\om \subset  \cup_j \d \om_j \cup N$, for some $N \subset \d \om$ with $\HH^{n}(N)=0$. Moreover,  $\HH^n|_{\d \om} \ll \hm^x_{\om}$  for any $x \in \om$.
\end{corollary}

}

\vvv

\subsection*{Acknowledgements} We warmly thank J. Azzam for his encouragement and several conversations pertaining to this work.

\vvv

\section{Background material}\label{sec:backgr}
\begin{itemize}
\item If $A,B\subset \R^{d}$, we let \[\dist(A,B)=\inf\{|x-y|:x\in A,y\in B\}, \;\; \dist(x,A)=\dist(\{x\},A),\]
\item $B(x,r)$ stands for the open ball of radius $r$ which is centered at $x$. We also denote by $\lambda B(x,r)=B(x,\lambda r)$.
\item  We will write $p \lesssim q$ if there is $C>0$ so that $p \leq C q$ and $p \lesssim_{M} q$ if the constant $C$ depends on the parameter $M$. We write $p \sim q$ to mean $p\lesssim q \lesssim p$ and define $p\sim_{M} q$ similarly. 
\item $G(d,m)$ is the Grassmannian manifold of all $m$-dimensional linear subspaces of $\R^d$.
\item We denote by $\pi_V: \R^d \to V$ the orthogonal projection on $V \in G(d, m)$.
\item $V^\perp \in G(d, d-m)$ is the orthogonal complement of $V \in G(d, m)$.
\item $f:E \subset \R^d \to \R^d$ is $L$-Lipschitz if  for all $x, y \in E$,
$$|f(x)-f(y)| \leq L |x-y|.$$
\item  $f:E\subset \R^d \to \R^d$ is  called $L$-bi-Lipschitz if for all $x, y \in E$,
$$L^{-1} |x-y| \leq |f(x)-f(y)| \leq L |x-y|.$$
\end{itemize}

\vvv

Let us recall now some elements of geometric measure theory following \cite{Mattila}. For $A\subset \R^{d}$, $s \in (0, d]$, { and $\delta \in (0, \infty]$,} we set
\[\HH^{s}_{\delta}(A)=\inf\left\{\sum (\diam A_i)^{s}: A\subset \bigcup_i A_i, \,\, \diam A_i<\delta \right\}.\]
Define the {\it $s$-Hausdorff measure} as $\HH^{s}(A)=\lim_{\delta\downarrow 0}\HH^{s}_{\delta}(A).$

\vv

\begin{definition}
Let $0< s<\infty$, $E \subset \R^d$ and $x \in \R^d$. The {\it upper  $s$-density} of $E$ at $x$ is defined by
\begin{align*}
\Theta^{*,s}(E, x) &= \limsup_{r \to 0} \frac{\HH^s(E \cap B(x,r))}{r^s}. 
\end{align*}
\end{definition}

\vv

{ Following \cite{ABHM16}, we give the following definition:}
\begin{definition}\label{def:WLADR}
We say that a set $E \subset \R^{d}$ satisfies the {\it weak lower Ahlfors-David $s$-regularity} condition ({WLADR}) at $x \in E$ if there exists $\rho_x>0$ such that
\begin{equation}
\inf_{(y, r) \in B(x,r) \times (0, \rho_x)} \frac{\HH^{s}(B(y,r) \cap E)}{r^{s}} >0.
\label{eq:w-regular}
\end{equation}
We will say that $E$ is WLADR  if $E$ is WLADR at $x$ for $\HH^s$-a.e. $x \in E$.
\end{definition}

\vv

\begin{definition}
 If $V \in G(d,m)$, $\xi \in \R^d$, and $s \in (0,1)$, we say that the set
\begin{align}\label{eq:tang.cone}
X(\xi, V, s) &= \{ x \in \R^d : \dist(x-\xi, V) < s|x - \xi| \} \notag\\
&= \{x\in \R^d:| \pi_{V^\perp}(x-\xi)|< s|x-\xi|\}
\end{align}
is a \textit{cone} around $\xi+V$ with vertex $\xi$ and aperture $s$.  For  $0<r<\infty$ we also set  
\begin{equation}\label{eq:tr.tang.cone}
X(\xi, V, s, r)=X(\xi, V, s) \cap B(\xi, r).
\end{equation}
to be the truncated cone at height $r$.

\vv

We can alternatively define the two-sided cone by
$$X(\xi, \nu, \alpha)= \{ x \in \R^d: | (x-\xi) \cdot \nu | > \alpha|x- \xi|\},$$
where $\alpha>0$ and $\nu \in \mathbb S^{d-1}:=\d B(0,1)$, while the one-sided cones are given by
\begin{align*}
X^+(\xi, \nu, \alpha) &= \{ x \in \R^d: (x-\xi) \cdot \nu> \alpha|x- \xi|\}, \\
X^-(\xi, \nu, \alpha) &= \{ x \in \R^d: (x-\xi) \cdot \nu < - \alpha|x- \xi|\} .
\end{align*}
We also denote the half-spaces that contain $\xi$ by
\begin{align*}
H^+(\xi, \nu) &= \{ x \in \R^d: (x-\xi) \cdot \nu > 0\}, \\
H^-(\xi, \nu) &= \{ x \in \R^d: (x-\xi) \cdot \nu < 0\},
\end{align*}
and notice that $X^\pm(\xi, \nu, \alpha)  \subset H^\pm(\xi, \nu)$. 

\end{definition}

\vv
{
\begin{remark}\label{rem:compl.cone}
Remark that ${^cX(\xi, V, s)}$, the complement ${X(\xi, V, s)}$, is actually the closure of the cone $X(\xi, V^\perp, \sqrt{1-s^2})$. In the case $V=\R^{d-1}$, we have  $\nu= \vec{e}_{n+1}= (0,0, \dots, 0,1)$ and $X(\xi, \R, \sqrt{1-s^2}) = X(\xi,  \vec{e}_{n+1}, s)$.
\end{remark}
 }
 \vv
 
\begin{definition}\label{def:appr.tang.pl}
Let $E \subset \R^d$, $\xi \in \R^d$, and $V \in G(d,m)$. For fixed $s\in (0, 1)$, we say that $V$ is an {\it $s$-approximate tangent $m$-plane} for $E$ at $\xi$ if $\Theta^{*, m}(E, \xi)>0$ and 
\begin{equation}\label{eq:cone.dens.0}
\lim_{r \to 0} \frac{\HH^m(E \cap B(\xi, r) \setminus X(\xi, V, s))}{r^m}=0.
\end{equation}
If this holds for all $s \in (0,1)$ then we just say that $V$ is an {\it approximate tangent $m$-plane} for $E$ at $\xi$. We write $\textup{ap-Tan}^m(E,\xi)$ for the set of all approximate tangent $m$-planes for $E$ at $\xi$.
\end{definition}

\vv

\begin{definition}\label{def:rec-lip}
If $E\subseteq \R^{d}$ is a Borel set, we say that $E$ is {\it $m$-rectifiable} if $\HH^m(E\backslash \bigcup_{i=1}^{\infty} \Gamma_{i})=0$ where $\Gamma_{i}=f_{i}(E_{i})$, $E_{i}\subseteq \R^{m}$, and $f_{i}:E_{i}\rightarrow \R^{d}$ is Lipschitz. 
\end{definition}

\vv

The criterion for rectifiability which will be most useful for us is the following:
\begin{theorem}{\cite[Theorem 15.19]{Mattila}}\label{def:rec-approx}
Let $E \subset \R^d$ be  $\HH^m$-measurable and $\HH^m|_E$ be locally finite. Then the following are equivalent:
\begin{enumerate}
\item $E$ is $m$-rectifiable.
\item For $\HH^m$ almost every point $\xi \in E$, there is a unique approximate tangent $m$-plane for $E$ at $\xi$.
\item  For $\HH^m$ almost every point $\xi \in E$, there is some approximate tangent $m$-plane for $E$ at $\xi$.
\end{enumerate}
\end{theorem}

\vv

\begin{definition}\label{def:BV}
A function $f \in L_{loc}^1(U)$ has {\it locally bounded variation} in an open set $U \subset \R^{n+1}$ and we write $f \in BV_{\loc}(U)$, if for each open set $V\Subset U$,
$$\sup \left\{ \int_V f \,\,\divv \phi \,\,d\mathcal L^{n+1}: \phi \in C^\infty_c(V;\R^{n+1}),\,\, |\phi| \leq 1 \right\} <\infty,$$
where $\mathcal L^{n+1}$ stands for the $(n+1)$-dimensional Lebesgue measure. An $\mathcal L^{n+1}$-measurable set $A \subset \R^{n+1}$ has {\it locally finite perimeter} in $U$ if $\chi_A \in BV_{loc}(U)$.
\end{definition}

\vv

\begin{theorem}
If $A \subset \R^{n+1}$ has {\it locally finite perimeter} in $U$, then there exists a Radon measure $\|\d A\|$ on $U$ and $\nu_A:U \to \R^{n+1}$ so that 
\begin{itemize}
\item[(i)] $|\nu_A(x)|=1$, for $\|\d A\|$-a.e. $x \in U$ and
\item[(ii)] $\int_A \divv \phi = \int \phi \cdot \nu_A \,d\,\|\d A\|$,
\end{itemize}
for all $\phi \in C^1_c(U; \R^{n+1})$.
\end{theorem}

\vv

\begin{definition}\label{def:reduced bdry}
Let $A$ be a set of locally finite perimeter in $\R^{n+1}$ and $x\in \R^{n+1}$. We say that $x \in \d^* A$, the {\it reduced boundary} of $A$, if
\begin{enumerate}
\item  $\|\d A\|(B(x,r))>0$, for all $r>0$,
\item $\lim_{r \to 0} \frac{1}{\|\d A\|(B(x,r))}\int_{B(x,r)}  \nu_A(y) \,d\|\d A\| = \nu_A(x)$, and
\item $|\nu_A(x)|=1$.
\end{enumerate}
\end{definition}

\vv

\begin{definition}\label{def:2.11}
For each $x \in \d^* A$ we define the {\it hyperplane}
$$H_A(x)=\left\{ y \in \R^{n+1}: \nu_A(x) \cdot (y-x) =0  \right\}$$
and the {\it half-spaces}
\begin{align*}
H_A^+(x)&=\left\{ y \in \R^{n+1}: \nu_A(x) \cdot (y-x) \geq 0  \right\},\\
H_A^-(x)&=\left\{ y \in \R^{n+1}: \nu_A(x) \cdot (y-x) \leq 0  \right\}.
\end{align*}
A unit  vector $\nu_A(x)$ is called the {\it measure theoretic unit outer normal} to $A$ at $x$ if
$$
\lim_{r \to 0} \frac{\mathcal L^{n+1} ( B(x,r) \cap A \cap H_A^+(x))}{r^{n+1}}=0
$$
and
$$
\lim_{r \to 0} \frac{\mathcal L^{n+1} ( (B(x,r) \setminus A) \cap H_A^-(x))}{r^{n+1}}=0.
$$
\end{definition}

\vv

\begin{definition}
Let $x \in \R^{n+1}$. We say that $x \in \d_*A$, the {\it measure theoretic boundary} of $A$, if 
$$
\limsup_{r \to 0} \frac{\mathcal L^{n+1} ( B(x,r) \cap A)}{r^{n+1}}>0
$$
and
$$
\limsup_{r \to 0} \frac{\mathcal L^{n+1} ( B(x,r) \setminus A)}{r^{n+1}}>0.
$$
\end{definition}

\vv
For the proof of the next theorem see  \cite[Theorem 2, p. 205]{EvansGariepy} and \cite[Lemma 1, p. 208]{EvansGariepy}..
\begin{theorem}
\label{thm:mt-reduced-bry}
If  $A \subset \R^{n+1}$ has locally finite perimeter, then the following hold: 
\begin{enumerate}
\item $\d^* A=\cup_{j\geq 1} K_j \cup N$, where $\|\d A\|(N)=0$ and $K_j$ is a compact subset of a $C^1$-hypersurface $S_j$.
\item  $\nu_A|_{S_j}$ is outer unit normal to $S_j$.
\item $\|\d A\|= \HH^n|_{\d^* A}$.
\item $\d^* A \subset \d_* A$ and $\HH^n(\d_* A \setminus \d^* A)=0$.
\item If $\HH^n(\d A \setminus \d_* A)=0$, then $\d A$ is $n$-rectifiable. 
\end{enumerate}
\end{theorem}
\vv

A useful criterion that allows us to determine whether a set has  locally finite perimeter, whose proof can be found in \cite[p.\ 222]{EvansGariepy}, is the following:
\begin{theorem}\label{thm:BV-Criterion}
If $A \subset \R^{n+1}$ is $\mathcal L^{n+1}$--measurable, then it has locally finite perimeter if and only if $\HH^n(K \cap \d_* A)<\infty$, for each compact set $K \subset \R^{n+1}$.
\end{theorem}

\vv

The results above in combination with  Theorem 5.6.5 and Lemma 5.5.4  in \cite{Z89} give the following proposition. 
\begin{proposition}\label{prop:locfinapptang}
If  $A \subset \R^{n+1}$ has locally finite perimeter and $\HH^n(\d A \setminus \d_* A)=0$, then $\HH^n$-a.e. $x \in \d A$ has a unique approximate tangent plane that coincides with $H_A(x)$ from Definition \ref{def:2.11}.
\end{proposition}

\vvv

\section{Proof of  Theorems}

We will first prove Theorem \ref{thm:structure}. By hypothesis, for a fixed $s \in (0,1)$ and $\xi \in \K:=\K_{m}(E)$ so that \eqref{eq:w-regular} holds,  there exists an $m$-plane $V_\xi$ passing through the origin so that 
\begin{equation}\label{eq:con.dens.zero}
\lim_{r \to 0} \frac{\HH^m(E \cap B(\xi, r) \cap {^cX}(\xi, V_\xi, s))}{r^m}=0.
\end{equation}
\vv

\begin{lemma}
There exists $r_\xi>0$ so that $^{c}X(\xi, V_\xi, 2s) \cap B(\xi, r_\xi) \cap E = \{\xi\} $.
\end{lemma}

\begin{proof}
Let $\rho_\xi$ be the radius from the definition of weak Ahlfors-David regularity at the point $\xi$. Let us assume that we can find a sequence of radii $r_i \to 0$ with $r_i < \rho_\xi$ so that for each $i\geq 1$, there exists $x_i \in {^{c}X}(\xi, V_\xi, 2s) \cap B(\xi, r_i) \cap E$. We may choose $r_i$ so that $c_0 r_i \leq |x_i-\xi| < r_i$, for a constant $c_0 \sim 1$ to be fixed momentarily. Moreover, we can find a constant $\delta \in (0,1)$ such that $\delta \sim_{s} 1$ and
\begin{equation} \label{eq:Bdelta-cone}
B(x_i, \delta r_i) \subset {^cX}(\xi, V_\xi, s) \cap B(\xi, 2 r_i).
\end{equation} 
Indeed, since $\pi_{V_\xi^\perp}$ is a linear 1-Lipschitz map, for any $y\in B(x_i, \delta r_i)$, it holds 
\begin{align*}
|\pi_{V_\xi^\perp} (y - \xi)| &= |\pi_{V_\xi^\perp} (x_i- \xi) + \pi_{V_\xi^\perp} (y - x_i) | \geq \left| 2s |x_i- \xi| -  |y - x_i| \right|\\
& \geq 2  s c_0\, r_i - \delta r_i \geq \frac{2 c_0 s - \delta}{1+\delta} |y -\xi| = s |y -\xi|,
\end{align*}
if we choose $\delta = \frac{(2c_0-1)s}{1+s}$ and $c_0 \in (1/2,1)$ so that $\delta \sim 1$. The fact that $B(x_i, \delta r_i) \subset B(\xi, 2r_i)$ is trivial. By \eqref{eq:Bdelta-cone} and \eqref{eq:con.dens.zero}, we have that
\begin{align*}
\inf_{(y, r) \in B(\xi,r) \times (0, \rho_\xi)} &\frac{\HH^{m}(E \cap B(y,r))}{r^{m}} \leq \lim_{r_i \to 0} \frac{\HH^m(E \cap B(x_i, \delta r_i))}{r_i^m} \\
\leq &\lim_{r_i \to 0}\frac{\HH^m( E \cap {^cX}(\xi, V_\xi, s) \cap B(\xi, 2 r_i) )}{r_i^m} = 0,
\end{align*}
which by the weak lower Ahlfors-David $m$-regularity of $E$ is a contradiction. This concludes our lemma.
\end{proof}

\vv

For $V, W \in G(n+1, n+1-m)$, we define $d(V, W)= \| \pi_V - \pi_W\|$, where $\|\cdot \|$ is the usual operator norm for linear maps. With this metric $G(n+1, n+1-m)$ is a compact metric space and thus, for any fixed number $\alpha \in (0, 1/3)$, there is a finite subset of $G(n+1, n+1-m)$, say $\P_m(\alpha)=\{V_j\}_{j =1}^{N(\alpha)}$,  such that the following holds: for any  $V  \in G(n+1, n+1-m)$, there exists $V_{j_0} \in \P_m$ so that $d(V, V_{j_0})<\alpha$.

\vv

\begin{lemma}\label{lem:3.2}
Assume that $\ve>0$. For any $\xi \in \K$, there exists $j=j(\xi, \ve)\in \bN$, such that $V_j \in \P_m$ and ${^cX}(\xi, V_j^\perp, 2s+\ve) \subset {^cX}(\xi, V_\xi, 2s)$. 
\end{lemma}

\begin{proof}
For fixed $\ve>0$ and $\xi \in \K$, there exists $V_j \in \P_m(\ve)$, so that $d(V_\xi^\perp, V_j)<\ve$. If $y \in {^cX}(\xi, V_j^\perp, 2s+\ve)$, we have that
\begin{align*}
|\pi_{V_\xi^\perp}(y-\xi)| &\geq |\pi_{V_j}(y-\xi)| - |(\pi_{V_j}-\pi_{V_\xi^\perp})(y-\xi)| \\
&\geq (2s+\ve - \ve)|y-\xi|\\
& = 2s |y-\xi|.
\end{align*}
This readily shows that $y \in {^cX}(\xi, V_\xi, 2s)$ and finishes our proof.
\end{proof}

\vv

We set
$$S_{j, k}= \left\{ \xi \in \K: j=j(\xi, s) \,\, \textup{and}\,\,{^cX}(\xi, V_j^\perp, 3s ) \cap B(\xi,  k^{-1}) \cap E = \{\xi\} \right\}.$$
Let us fix $j \in \{1, 2, \dots, N(s)\}$ and $k\in \bN$. Since $S_{j, k}$ is separable it has  a countably dense subset $\{x_\ell\}_{\ell= 1}^\infty$. Therefore, for each $\xi \in S_{j, k}$,  there exists $\ell$ so that $|x_\ell - \xi|< (10 k )^{-1}$. Notice that there might be more than one $\ell$ for each $\xi$. Although, to any fixed $\xi \in S_{j, k}$, we assign once and for all a unique $\ell(\xi)$ with the requirement that $|x_{\ell(\xi)} - \xi|< (10 k )^{-1}$. If we set
$$S_{j, k, \ell}=\left\{ \xi \in S_{j,k}: \ell(\xi)=\ell \right\},$$
then we get that 
\begin{equation}\label{eq:K-decomp}
\K = \bigcup_{j} \bigcup_{k} \bigcup_{\ell} S_{j, k, \ell}.
\end{equation}

\vv

Fix now $j, k$ and $\ell$ so that $S_{j,k,\ell}\neq \emptyset$ and denote $S=S_{j,k,\ell}$.  Without loss of generality we may assume that $V_j=\R^{n+1-m}$ and $V_j^\perp=\R^{m}$ since projections are invariant under rotations.

\vv

The following lemma is contained in the proof of Lemma 15.13 in \cite{Mattila} but we provide a proof for completeness. 
\begin{lemma}\label{lem:K-Lipgraph}
$S$ is contained in the graph of a (rotated) $(\sqrt{1-(3s)^2})^{-1}$-Lipschitz function $\psi : \R^m \to \R^{n+1-m}$.
\end{lemma}

\begin{proof}
 Let $\xi \in S$. Note that if $|\pi_{\R^{m}}(\xi) - \pi_{\R^{m}}(\xi')|< \sqrt{1-(3s)^2} |\xi-\xi'|$ and $|\xi -\xi'|<k^{-1}$, then $\xi' \in X(\xi, \R^{n+1-m}, \sqrt{1-(3s)^2}) \cap B \left(\xi, k^{-1} \right)$, or equivalently, $\xi' \in {^cX}(\xi, \R^{m}, 3s) \cap B(\xi, k^{-1})$. By hypothesis, this means that $\xi' \not \in E$ and thus, $|\pi_{\R^{m}}(x) - \pi_{\R^{m}}(\xi)| \geq \sqrt{1-(3s)^2} |x-\xi|$, for any $x, \xi \in S$. This implies that $\pi_{\R^{m}}|_S$ is a $\sqrt{1-(3s)^2}$-bi-Lipschitz map with $(\sqrt{1-(3s)^2})^{-1}$-Lipschitz inverse 
$$\widetilde f=(\pi_{\R^{m}}|_S)^{-1}: \pi_{\R^{m}}(S) \to \R^{n+1}.$$
 Note that  $S = \widetilde f (\pi_{\R^{m}} (S))$. By Kirszbraun's theorem, we may extend $\widetilde f$ to a globally defined $(\sqrt{1-(3s)^2})^{-1}$-Lipschitz function $ f: \R^m \to \R^{n+1}$ with  $f|_{\pi_{\R^{m}}(S)} = \widetilde f$. If we set $\psi= \pi_{\R^{n+1-m}} \circ f$, then it is clear that $\psi$ is $(\sqrt{1-(3s)^2})^{-1}$-Lipschitz and  every $x \in S$ belongs to the graph $\Gamma_\psi:=\{ \left( y, \psi (y)\right): y \in \R^{m}\}$. 
\end{proof}

\vv

Theorem \ref{thm:structure} readily follows from the above lemmas. Note here that Lemma \ref{lem:K-Lipgraph} is of independent interest since we did not assume that the $m$-Hausdorff measure on $E$ is locally finite. In the special case $m=n$ and towards the construction of Lipschitz domains, we shall give a different proof which is also based on the cone property of the points with approximate tangents. Recall that (possibly)  after   rotation, it is enough to consider the cones $X^\pm(x, \vec{e}_{n+1}, 3s)$ over points $x \in S$.

\vv

\begin{lemma}\label{lem:K-Lipgraph-co1}
Let $m=n$ and $S$ be as above. Then $S$ is contained in the graphs $\G_{\psi_\pm}$ of the Lipschitz functions $\psi_\pm : \R^{n} \to \R$ given by
$$\psi_+(z)=\inf_{x \in {S}} f^+_x(z) \,\, \textup{and}\,\, \psi_-(z)=\sup_{x \in {S}} f^-_x(z),$$ 
where $f^\pm_x$ stand for the graph of the boundary of the cone $X^\pm(x, \vec{e}_{n+1}, 3s)$. 
\end{lemma}

\begin{proof}
We only show the one case since the proof for the other is identical. Recall that we proved that each $x \in {S}$ is a cone point and note that the boundary of the cone $X^+(x, \vec{e}_{n+1}, 3s)$ is given by a Lipschitz graph $(z, f^+_x(z)) \in \R^{n} \times \R$. It is easy to see that if we define the Lipschitz function $\psi_+(z)=\inf_{x \in {S}}  f^+_x(z)$, then ${S}$ is contained in the Lipschitz graph $(z, \psi^+(z))$. 
\end{proof}

\vv

\begin{lemma}\label{lem:Lip-domains}
Let $m=n$, $0<s\leq\frac{1}{\sqrt{90}}$, and $S$ be as above. Then there exist   Lipschitz domains $\Omega^+_{S}$ and $\Omega^-_{S}$ such that  ${S} \subset \d \Omega^\pm_{S}$ and $\d \om \cap \Omega^\pm_{S} = \emptyset$.
\end{lemma}

\begin{proof}
Without loss of generality we may assume that $x_\ell=0 \in  S$ and  $\nu_{x_\ell}=\vec{e}_{n+1}$.  Let $x^\pm$ be the endpoint of the vector $\pm \frac{1}{5k} \vec{e}_{n+1}$ and set 
$$W^\pm=\{y \in \R^{n+1} : (x^\pm-y)\cdot \vec{e}_{n+1}=0\}.$$
Note that $W^+$ and $W^-$ are two $n$-planes perpendicular to $\vec{e}_{n+1}$ that respectively contain $x^+$ and $x^-$. 
We now set  $C_k:=C(\vec{e}_{n+1}, (10 k)^{-1})$ to be the infinite cylinder with axis $\vec{e}_{n+1}$ and radius $(10 k)^{-1}$. It is clear that if $y \in C_k$, then $|\pi_{\R^n}(y)|<(10 k)^{-1}$.

 If $\xi \in S$, then for any $z \in \d X(\xi,  \vec{e}_{n+1}, 3s) \cap C_k$, by Remark \ref{rem:compl.cone}, we have that 
\begin{align}
|\pi_{\R}(z -\xi)|&= \frac{3s |\pi_{\R^{n}}(z- \xi)|}{\sqrt{1-(3s)^2}} \leq \frac{3s}{\sqrt{1-(3s)^2}} \left(|\pi_{\R^{n}}(z)|+ |\pi_{\R^{n}}(\xi)|\right)\notag \\
&< \frac{3s}{5k \sqrt{1-(3s)^2}},\label{eq:proj.x-xi}
\end{align}
which, for $s\leq 1/\sqrt{90}<1/3$, implies that $|\pi_{\R}(z -\xi)|\leq \frac{1}{15 k}$. Moreover, since 
$|\pi_{\R}(\xi)| \leq | \xi | = | \xi - x_{\ell}| \leq (10 k)^{-1}$, we infer that $|\pi_{\R}(z)| \leq (6 k)^{-1}$. Thus, $|\pi_{\R}(x^\pm)-\pi_{\R}(z)| > (30 k)^{-1}$.

If we define $\Omega_{{S}}^\pm$ to be the part of $C_k$ which is contained between $W^\pm$ and the Lipschitz graph $\Gamma_{\psi_\pm}$ of Lemma \ref{lem:K-Lipgraph-co1}, from the considerations above, it is clear  that it is connected. Therefore, by construction, $\Omega_{{S}}^\pm$ is a bounded  Lipschitz domain satisfying $S \subset \d\om^\pm_S$ and $\d \om \cap \om_{{S}}^+ = \emptyset$, which concludes our proof.
\end{proof}

\vv

Theorem \ref{thm:Lip-dom} follows  from Lemma \ref{lem:Lip-domains}.
\begin{remark}\label{rem:F subset K}
Note that the above considerations can be repeated for any ${\mathcal F} \subset \K$ producing a countable union of $S_{j,k,\ell}(\mathcal F)$ that  exhausts ${\mathcal F}$ so that Lemma \ref{lem:Lip-domains} holds for each such $S_{j,k,\ell}(\mathcal F)$.  In the next lemma we denote by $S_{\mathcal F}$ a fixed $S_{j,k,\ell}(\mathcal F)$.
\end{remark}
\vv

{
\begin{lemma}\label{lem:Lip-dom_inOmega}
Let $\Omega \subset \R^{n+1}$ be an open and connected set and let $0<s\leq\frac{1}{\sqrt{90}}$. If $\d^\star \om$ is as in Definition \ref{def:thick} and
$S_{ \d^\star \om} \subset\d^\star \om$ is  as in  Remark \ref{rem:F subset K} for $\mathcal F =\d^\star \om$,  there exist  bounded Lipschitz domains  $ \Omega_{\star}^+$ and $  \Omega_{\star}^-$ such that $ \Omega_{\star}^\pm \subset \Omega$  and $S_{ \d^\star \om}  \subset \d \Omega_{\star}^+ \cup \d \Omega_{\star}^-$. 
\end{lemma}

\begin{proof}
We can prove this lemma repeating the proof of Lemma \ref{lem:Lip-domains} for  ${S}=S_{\d^\star \om}$ and using the  definition of $\d^\star \om$. Indeed, by construction,  we may assume without loss of generality that  $X^\pm(\xi, \vec{e}_{n+1}, 3s) \subset X^\pm(\xi, \nu_x^+, 2s)$,  for any  $\xi \in S$, and $x_\ell=0$. We will argue by contradiction. To this end, fix $\xi \in S$ and assume that 
$$
X^+(\xi, \vec{e}_{n+1}, 3s) \cup X^-(\xi, \vec{e}_{n+1}, 3s) \cap B(\xi, k^{-1})  \subset \mathbb R^{n+1} \setminus \om.
$$ 
By the  definition of $\d^\star \om$, we have that either \eqref{eq:thick1} or \eqref{eq:thick2} holds.  Then there exists a constant $c(s,n) \in (0,1)$ such that
\begin{align*}
c(s,n) \leq \liminf_{r \to 0^+} \frac{|B(\xi,r)\cap X^\pm(\xi, \vec{e}_{n+1}, 3s)|}{r^{n+1}} &\leq \liminf_{r \to 0^+} \frac{|B(\xi,r)\cap X^\pm(\xi, \nu_\xi^+, 2s) \setminus \om|}{r^{n+1}}\\
& \leq \liminf_{r \to 0^+} \frac{|B(\xi,r)\cap H^+(\xi, \nu_\xi^\pm) \setminus \om|}{r^{n+1}}=0,
\end{align*}
 which is a contradiction. Here  we used that $X^\pm(\xi, \nu_\xi^+, 2s) \subset H^+(\xi, \nu_\xi^\pm)$. Thus,  for any $\xi \in S$, at least one of the truncated cones  $X^\pm(\xi, \vec{e}_{n+1}, 3s)\cap B(\xi, k^{-1})$ is contained in $\om$.  If we set  
$$
S^\pm:= \{ \xi \in S : X^\pm(\xi, \vec{e}_{n+1}, 3s)\cap B(\xi, k^{-1}) \subset \om\},
$$ 
  in view of  Lemma \ref{lem:K-Lipgraph-co1} and the proof of Lemma \ref{lem:Lip-domains}, we can construct bounded Lipschitz domains  $ \Omega_{\star}^\pm\subset \Omega$   such that  $S^\pm  \subset \d \Omega_{\star}^\pm$.  If $\Omega_{\star}^+ \cap \Omega_{\star}^- \neq \emptyset $, then $\om_\star =  \Omega_{\star}^+ \cup \Omega_{\star}^-$ is a bounded Lipschitz domain. Otherwise, we have two disjoint domains. We omit the details.
\end{proof}

\vv
}

\begin{proof}[Proof of Theorem \ref{thm:Lip-dom-int}] The construction of interior bounded Lipschitz subdomains follows from Lemma \ref{lem:Lip-dom_inOmega}. We will only prove $\HH^n \ll \hm_\om^x$ on $\d^\star\om$.  To this end, let $G \subset \d^\star \om$ be a Borel set so that $\HH^n(G)>0$. Then, there exists a Lipschitz domain $\om_j \subset \om$ so that $\HH^n(G \cap \om_j)>0$. By Dahlberg's theorem, for any $p \in \om_j$, it holds that $\hm^p_{\om_j}(G)>0$, which, in turn, by maximum principle and the connectivity of $\om$, implies that $\hm^x_{\om}(G)>0$ for any $x \in \om$,  concluding our proof. 
\end{proof}

\vv
{

\begin{proof}[Proof of Corollary \ref{cor:w-abscont-rect}]
This is a direct consequence of  Definitions \ref{def:WLADR} and \ref{def:2.11}, Proposition \ref{prop:locfinapptang},  and Theorem \ref{thm:Lip-dom-int}.
\end{proof}
}

\end{document}